\documentclass[11pt]{amsart}

\usepackage{epsfig}
\usepackage{amsmath}
\usepackage{amssymb,mathrsfs}
\usepackage{footmisc} %%% to allow reference to footnotes.

\usepackage[all]{xy}

\usepackage{tikz-cd}

\usepackage{color}
\usepackage{tabularx}
\newcolumntype{C}{>{\centering\arraybackslash}X}

\def\comment#1{{\sf{[#1]}}}
\def\Z{{\mathbb Z}}
\def\Q{{\mathbb Q}}

\def\R{{\mathbb R}}
\def\C{{\mathbb C}}
\def\P{{\mathbb P}}

\def\bH{{\mathbb H}}
\def\LL{{\mathbb L}}
\def\V{{\mathbb V}}

\def\kk{{\Bbbk}}		%% {{\mathbbm k}}

\def\cG{{\mathcal G}}

\def\cJ{{\mathcal J}}

\def\sH{{\mathscr H}}
\def\sM{{\mathscr M}}
\def\sU{{\mathscr U}}
\def\sX{{\mathscr X}}

\def\g{{\mathfrak g}}

\def\k{{\mathfrak k}}

\def\p{{\mathfrak p}}

\def\u{{\mathfrak u}}

\def\w{{\omega}}

\def\G{{\Gamma}}

\def\ee{\mathbf{e}}
\def\ff{\mathbf{f}}
\def\hh{\mathbf{h}}
\def\ww{\mathbf{w}}
\def\zz{\mathbf{z}}
\def\pp{\mathbf{p}}

\def\bmu{\boldsymbol{\mu}}

\def\deltatilde{\tilde{\delta}}
\def\psihat{\hat{\psi}}
\def\rhohat{\hat{\rho}}
\def\lambdahat{\tilde{\lambda}}

\def\H{\widetilde{H}}
\def\Htilde{\widetilde{H}}
\def\Fhat{\widehat{F}}
\def\Mhat{{\widehat{M}}}
\def\Khat{\widehat{K}}

\def\sXhat{\widehat{\sX}}

\def\varphitilde{{\tilde{\varphi}}}

\def\MHS{{\mathsf{MHS}}}

\def\To{\longrightarrow}
\def\bdot{{\bullet}}

\def\blank{{{\phantom{x}}}}		%%		{\phantom{x}}

\def\ffs{{/\negthickspace /}}

\def\Sp{{\mathrm{Sp}}}

\def\GL{{\mathrm{GL}}}
\def\SO{{\mathrm{SO}}}
\def\SU{{\mathrm{SU}}}
\def\Gm{{\mathbb{G}_m}}

\def\sL{{\mathfrak{sl}}}

\def\ho{{\mathrm{ho}}}
\def\rat{{(0)}}
\def\un{{\mathrm{un}}}

\def\bil{{\langle\phantom{x},\phantom{x}\rangle}}

\newcommand\im{\operatorname{im}} 
\newcommand\id{\operatorname{id}}

\newcommand\Hom{\operatorname{Hom}}

\newcommand\End{\operatorname{End}}
\newcommand\Aut{\operatorname{Aut}}
\newcommand\Diff{\operatorname{Diff}}

\newcommand\Gr{\operatorname{Gr}}

\newcommand\Sym{\operatorname{Sym}}
\newcommand\rank{\operatorname{rank}}

\numberwithin{equation}{section}

\newtheorem{theorem}{Theorem}[section]
\newtheorem{lemma}[theorem]{Lemma}
\newtheorem{proposition}[theorem]{Proposition}
\newtheorem{corollary}[theorem]{Corollary}
\newtheorem{bigtheorem}{Theorem}
\newtheorem{bigcorollary}[bigtheorem]{Corollary}

\theoremstyle{definition}
\newtheorem{definition}[theorem]{Definition}
\newtheorem{example}[theorem]{Example}

\theoremstyle{remark}
\newtheorem{remark}[theorem]{Remark}

\newtheorem{question}[theorem]{Question}
\newtheorem{problem}[theorem]{Problem}

\newtheorem{conjecture}[theorem]{Conjecture}

\begin{document}

\title{Mapping Class Groups of Simply Connected K\"ahler Manifolds}

\author{Richard Hain}
\address{Department of Mathematics\\ Duke University\\
Durham, NC 27708-0320}
\email{hain@math.duke.edu}

%\thanks{Supported in part by grant DMS-1406420 from the National Science Foundation.}
\thanks{ORCID: {\sf 0000-0002-7009-6971}}

\date{\today}



%\begin{abstract}
%\end{abstract}

\maketitle

%%\tableofcontentst

\section{Introduction}

The {\em mapping class group} $\G_M$ of a closed orientable manifold $M$ is the group of isotopy classes of orientation preserving diffeomorphisms of $M$:
$$
\G_M := \pi_0 \Diff^+ M.
$$
A mapping class is the isotopy class of a diffeomorphism. The {\em Torelli group} $T_M$ of $M$ is the subgroup consisting of the mapping classes that act trivially on the homology of $M$:
$$
T_M := \ker\{\G_M \to \Aut H_\bdot(M;\Z)\}.
$$
Denote the image of $\G_M \to \Aut H_\bdot(M;\Z)$ by $S_M$. The mapping class group $\G_M$ is an extension
$$
1 \to T_M \to \G_M \to S_M \to 1.
$$

Mapping class groups of compact oriented surfaces play a fundamental role in the study of projective algebraic curves and their moduli and one can ask whether this remains true in higher dimensions. It appears that, at present, not much is known and that the subject is still in its infancy. The goal of this paper is to survey and supplement known results on mapping class and Torelli groups of simply connected compact K\"ahler manifolds and their relation to fundamental groups of moduli spaces of projective manifolds. To help stimulate the study of such questions about mapping class groups, we have assembled a list of questions and open problems. It can be found in Section~\ref{sec:future}.

General results of Dennis Sullivan \cite{sullivan}, stated in Section~\ref{sec:sullivan}, imply that if $M$ is simply connected with $\dim_\R M \ge 5$, then $S_M$ is an arithmetic subgroup of an algebraic subgroup of the group of automorphisms of the cohomology ring $H^\bdot(M;\Q)$ that fix the rational Pontryagin classes and that $T_M$ mod a finite group is a lattice in a unipotent $\Q$-group. We combine Sullivan's result with some Hodge theory to show (Theorem~\ref{thm:kahler_coho}) that if $(M,\w)$ is simply connected compact K\"ahler manifold of complex dimension $>2$, then the group $G$ of automorphisms of $H^\bdot(M;\R)$ that fix $\w$ and the Pontryagin classes is reductive and that $S_M$ is a lattice in $G(\R)$.

The problem then is to understand $T_M$ and, when $M$ underlies a complex projective algebraic manifold, the relationship between $\G_M$ and the orbifold fundamental group of the moduli space that parameterized the projective algebraic structures on $M$.
Denote the group of homotopy classes of ``orientation preserving'' homotopy equivalences of $M$ with itself by $\ho E_M^+$. The first wrinkle is that, when $\dim_\R M > 2$, the natural homomorphism
$$
\G_M \to \ho E_M^+
$$
may not be injective, unlike in the classical case (i.e., $\dim_\R M = 2$) where it is an isomorphism. As pointed out by Sullivan \cite{sullivan} (and explained in Section~\ref{sec:distortion}) the kernel mod torsion is a full lattice in the ``distortion group'', which is a quotient
$$
D_M = \big[\bigoplus_{4j \le \dim M} H^{4j-1}(M:\Q)\big]/I
$$
of the rational cohomology of $M$ in degrees congruent to 3 mod 4. It detects elements of the kernel by measuring how they ``distort'' the Pontryagin classes of $M$.

Denote the elements of $\ho E_M^+$ that act trivially on $H^\bdot(M)$ by $\ho T_M$. A first step towards understanding the Torelli group $T_M$ is to look for generalizations of the Johnson homomorphism \cite{johnson:homom} to higher dimensions. Because of the distortion group, there are two steps. The first is to construct $S_M$ invariant homomorphisms that detect $\ho T_M$; the second is to compute $D_M$ and its image in $H_1(T_M;\Q)$.

After reviewing some classical homotopy theory in Section~\ref{sec:adams}, we give a general construction (for simply connected $M$) of a Johnson type homomorphism from $H_1(\ho T_M)$ into a subquotient of the homology of the space of based loops in $M$. This is done in Section~\ref{sec:johnson}. In the case of simply connected K\"ahler 3-folds, this construction and Sullivan's results give the following result, which should be regarded as a prototype of similar results that should hold in higher dimensions.

\begin{bigtheorem}
\label{thm:main}
If $M$ is a simply connected compact K\"ahler 3-fold, there are $S_M$-invariant surjections
\begin{equation}
\label{eqn:tau}
H_1(T_M;\Q) \twoheadrightarrow H_1(\ho T_M;\Q) \overset{\tau_M}{\twoheadrightarrow} \Hom(H_3(M;\Q),\Sym^2 H_2(M;\Q)/\im\Delta),
\end{equation}
where $\Sym^2$ denotes the symmetric square and $\Delta : H_4(M;\Q) \to S^2 H_2(M;\Q)$ is the dual of the cup product.
\end{bigtheorem}

The homomorphism $\tau_M$ is constructed from the action of $\ho T_M$ on the well known exact sequence
$$
0 \to \Sym^2 H_2(M;\Q)/\im\Delta \to \pi_3(M,x_o)\otimes\Q \to H_3(M;\Q) \to 0.
$$
This sequence is a higher dimensional analogue of the sequence
$$
1 \to L^2\pi_1(M,x_o)/L^3 \to \pi_1(M,x_o)/L^3 \to H_1(M) \to 0
$$
which was used by Dennis Johnson \cite{johnson:homom} to construct the classical Johnson homomorphism
$$
H_1(T_M) \to \Lambda^3 H_1(M) \subset  \Hom(H_1(M),L^2\pi_1(M,x_o)/L^3),
$$
when $M$ is a compact oriented surface of genus $\ge 3$ and where $L^k\pi_1(M,x_o)$ denotes the $k$th term of the lower central series of its fundamental group. Johnson \cite{johnson:h1} proved that the kernel is a finite 2-group.

The previous result gives the following lower bound on the rank of $H_1(T_M)$ in terms of the Betti numbers $b_j$ of $M$.

\begin{bigcorollary}
\label{cor:bound}
If $M$ is a simply connected K\"ahler 3-fold, then
$$
\rank H_1(\ho T_M) \ge b_3\binom{b_2+1}{2} - b_2b_3 = \frac{(b_2-1)b_2b_3}{2}.
$$
In particular, if $b_2 > 1$ and $b_3>0$, then $H_1(T_M;\Q) \neq 0$.
\end{bigcorollary}

This bound is vacuous when $b_2=1$, such as when $M$ is a smooth hypersurface in $\P^4$. In this case, Kreck and Su \cite{kreck-su:2} compute the structure of $T_M$ and its abelianization. A consequence of their main result is:

\begin{bigtheorem}[Kreck--Su]
\label{thm:kreck-su}
If $M$ is a simply connected compact K\"ahler 3-fold with $b_2 = 1$, then distortion defines a homomorphism
$$
T_M \to H^3(M;\Q)
$$
whose image is a lattice of full rank and whose kernel is finite.
\end{bigtheorem}

Their result implies that for such manifolds, $H_1(\ho T_M;\Q)$ vanishes so that $H_1(T_M;\Q)$ is isomorphic to the distortion group $D_M$.

In the case of complex projective curves of genus $\ge 2$, the mapping class group $\G_M$ is the orbifold fundamental group of the moduli space $\sM_g$ of smooth projective curves of genus $g$. It is natural to ask whether this relationship persists in higher dimensions. In Section~\ref{sec:hypersurfaces} we explain how a vanishing result for the distortion homomorphism for hypersurfaces combined with a result of Carlson and Toledo \cite{carlson-toledo} imply that for hypersurfaces of degree $\ge 3$ and dimension congruent to 3 mod 4, the fundamental group of the corresponding moduli space is far from being the mapping class group of the underlying manifold.

\begin{bigtheorem}
\label{thm:hypersurfaces}
Denote the moduli stack of smooth hypersurfaces of degree $d$ in $\P^{n+1}$ by $\sH_{n,d}$. If $d\ge 3$ and $n\ge 3$, then the kernel of the monodromy representation
$$
\lambda_M : \pi_1(\sH_{n,d},[M]) \to \G_M
$$
contains a non-abelian free group. When $n\equiv 3\bmod 4$, its image has infinite index.
\end{bigtheorem}

The lower bound on the kernel is due to Carlson--Toledo \cite{carlson-toledo} (see Theorem~\ref{thm:carlson_toledo}), while the constraint on the image is Theorem~\ref{thm:vanishing}. The images of $\pi_1(\sH_{n,d},[M])$ and $\G_M$ in $\Aut H^n(X)$ are commensurable by a result of Beauville \cite{beauville}. I do not know whether or not the image of $\lambda_M$ has infinite index when $n$ is not congruent to 3 mod 4.

Many questions and problems arose as I was writing these notes. A selection of them is presented in Section~\ref{sec:future}. Among them is a proposal for using relative mapping class groups to study fundamental groups of moduli spaces of hypersurfaces to circumvent the issues that arise in Theorem~\ref{thm:hypersurfaces}. In addition, Hodge theory should play a role in the study of fundamental groups of moduli spaces of projective varieties and the associated mapping class group as it does in the classical case. For this reason, the list contains several questions about the Hodge theory of fundamental groups of moduli stacks of projective manifolds and its relation to the (conjectural) Hodge theory of the corresponding mapping class groups.

\bigskip

The origin of this paper lies in the counter examples of Kreck and Su \cite{kreck-su} to a ``Theorem''  of Verbitsky \cite[Thm.~3.4]{verbitsky}, which incorrectly asserted that the Torelli group of a simply connected compact K\"ahler manifold is finite. Fortunately this error does not affect the other major results in Verbitsky's paper. (See the Erratum \cite{verbitsky:erratum} and \cite{looijenga} for details.) This paper is a much expanded version of a note in which Theorem~\ref{thm:main} was proved and whose goal was to give a second family of counter examples.

\bigskip
\noindent{\bf Acknowledgments:} I would like to thank Alexander Berglund for pointing out the example in Section~\ref{sec:coho_autos}. I am also grateful to the referee for their careful reading of the manuscript.

\section{Quick Review of the Classical Johnson Homomorphism}

The setting for the classical Johnson homomorphism is where $M$ is a compact oriented surface of genus $\ge 3$. In this case, one has the short exact sequence
\begin{equation}
\label{eqn:extn}
0 \to I_{x_o}^2/I_{x_o}^3 \to I_{x_o}/I_{x_o}^3 \to I_{x_o}/I_{x_o}^2 \to 0
\end{equation}
where $I_{x_o}$ denotes the kernel of the augmentation $\Z\pi_1(M,x_o) \to \Z$. There are natural isomorphisms
$$
I_{x_o}/I_{x_o}^2 \cong H_1(M) \text{ and } I_{x_o}^2/I_{x_o}^3 \cong H_1(M)^{\otimes 2}/\theta
$$
where
$$
\theta = \sum_{j=1}^g (a_j \otimes b_j - b_j \otimes a_j)
$$
and $a_1,\dots,a_g,b_1,\dots,b_g$ is a symplectic basis of $I_{x_o}/I_{x_o}^2$.

Each diffeomorphism $\phi$ of $(X,x_o)$ induces an automorphism of $I_{x_o}/I_{x_o}^3$. If $\phi \in T_M$, it acts trivially on the kernel and cokernel of (\ref{eqn:extn}), so that
$$
\phi_\ast - \id \in \Hom(I_{x_o}/I_{x_o}^2,I_{x_o}^2/I_{x_o}^3) \cong \Hom(H_1(M), H_1(M)^{\otimes 2}/\theta).
$$
The map
$$
\tau_M : T_M \to \Hom(H_1(M), H_1(M)^{\otimes 2}/\theta)
$$
is a homomorphism and is one of the incarnations of the Johnson homomorphism. It is invariant under the image $S_M$ of the mapping class group of $(M,x_o)$ in $\Aut H_1(M)$, which is isomorphic to $\Sp_g(\Z)$. In this case, the homomorphism $\tau_M$ is not surjective. Rather its image is isomorphic to a copy of $\Lambda^3 H_1(M)$ inside the target. See \cite{johnson:survey} or \cite{hain:msri} for complete statements and details. One of Johnson's basic results \cite{johnson:h1} asserts that $\tau_M$ induces an isomorphism of $H_1(T_M)/\text{2-torsion}$ with the image of $\tau_M$.

\section{Some Classical Results from Homotopy Theory}

The analogue of the group algebra $\kk \pi_1(X,x_o)$ for a simply connected space $X$ is the homology of its loop space $L_{x_o}X$. Both are cocommutative Hopf algebras whose structure is related to that of $\pi_1(X,x_o)$ in the first case and to $\pi_\bdot(X,x_o)$ in the second. To explain this, we review some classical homotopy theory.

Suppose that $X$ is a simply connected topological space. Fix a point $x_o \in X$. Denote the space of loops $(S^1,1) \to (X,x_o)$ in $X$ based at $x_o$ by $L_{x_o} X$. There is a standard isomorphism
\begin{equation}
\label{eqn:iso}
\pi_j(X,x_o) \cong \pi_{j-1}(L_{x_o}X,c_o)
\end{equation}
for all $j > 0$, where $c_o$ denotes the constant loop at $x_o$. The collection of higher homotopy groups forms a graded Lie algebra over $\Z$, where $\pi_j(X,x_o)$ is placed in degree $j-1$. The Lie bracket is called the {\em Whitehead product}. The definition can be found in \cite[\S X.7]{whitehead}. The loop space homology has a multiplication (the Pontryagin product), which is induced by multiplication of loops. The Hurewicz homomorphism
$$
h : \pi_\bdot (L_{x_o}X,c_o) \to H_\bdot(L_{x_o} X)
$$
takes the Whitehead product to the Pontryagin commutator \cite[Thm.~X.7.10]{whitehead}\footnote{We multiply Whitehead's definition of the Whitehead product $[\alpha,\beta]$ by $(-1)^{|\alpha|}$ so that this is the case.}
$$
h([\alpha,\beta]) = h(\alpha)h(\beta) - (-1)^{|\alpha||\beta|} h(\beta)h(\alpha),
$$
where $|\alpha|$ denotes the degree of $\alpha$, etc.

When $\kk$ is a field of characteristic zero, the Hurewicz homomorphism
$$
\pi_\bdot(L_{x_o}X,c_o) \otimes \kk \to H_\bdot(L_{x_o} X;\kk)
$$
is injective and induces an isomorphism
$$
U \pi_\bdot(L_{x_o}X,c_o)\otimes \kk \overset{\simeq}{\To} H_\bdot(L_{x_o} X;\kk)
$$
of the universal enveloping algebra of $\pi_\bdot(L_{x_o} X)\otimes \kk$ with the loop space homology. A proof can be found in the appendix of \cite{milnor-moore}. The Lie algebra $\pi_\bdot(L_{x_o}X,c_o)\otimes \kk$ can be recovered from the loop space homology by taking primitives:
$$
\pi_\bdot(L_{x_o}X,c_o)\otimes \kk \cong PH_\bdot(L_{x_o} X;\kk) = \{u \in H_\bdot(L_{x_o} X;\kk):\Delta u = 1\otimes u + u \otimes 1\}.
$$

\section{Adams' Spectral Sequence for Loop Space Homology}
\label{sec:adams}

In this section, $\kk$ will be a field of characteristic 0 and $X$ will denote a simply connected space whose $\kk$ Betti numbers $\dim H_j(X;\kk)$ are finite. Fix a base point $x_o$ of $X$. In \cite{adams} (see also \cite[Appendix]{chen}), Adams gives a method for computing the homology of $L_{x_o} X$. For us, the important point is that he constructs a second quadrant spectral sequence
$$
E^1_{-s,t} = [\H_\bdot(X;\kk)^{\otimes s}]^{t-s} \implies H_{t-s}(L_{x_o} X;\kk),
$$
which converges to the homology of the loop space. The $E^1$-term is the tensor algebra
$$
T(\H_\bdot(X;\kk)[1])
$$
on the reduced homology of $X$ lowered in degree by 1.\footnote{To be clear, the degree $j$ component of $\H_\bdot(X;\kk)[1]$ is $\H_{j+1}(X;\kk)$.} The differential $\partial_1$ is the derivation of degree $-1$ defined by the map
$$
\partial_1 :  \H_\bdot(X;\kk) \to \H_\bdot(X;\kk)^{\otimes 2}
$$
dual to the map
$$
\xymatrix{
\H^\bdot(X;\kk)^{\otimes 2} \ar[r]^{J\otimes \id} & \H^\bdot(X;\kk)^{\otimes 2} \ar[r]^{\text{cup}} & \H^\bdot(X;\kk)
}
$$
where $J(u) = (-1)^{|u|} u$.

\begin{remark}
The original version of the spectral sequence is with $\Z$ coefficients. In that case, $\H(X)^{\otimes s}$ has to be replaced by $H_\bdot((X,x_o)^s)$. The differential is not so easily described in the presence of torsion.
\end{remark}

\begin{remark}
This spectral sequence is dual to the Eilenberg--Moore spectral sequence, which converges to $H^\bdot(L_{x_o}X;\kk)$. For a proof, see \cite[\S2.5]{chen}. This implies that (with coefficients $\kk$) the spectral sequence above is a spectral sequence in the category Hopf algebras, where the coproduct on the $E^1$ term $T(\H_\bdot(X;\kk)[1])$ is defined by making each element of $\H_\bdot(X;\kk)$ primitive. This implies that it can be restricted to its set of primitive elements to obtain a spectral sequence that converges to the Lie algebra $\pi_{\bdot}(X,x_o)\otimes \kk$, which we henceforth denote $\pi_{\bdot}(X,x_o)_\kk$.
\end{remark}

\begin{remark}
When $X$ is formal in the sense of rational homotopy theory, the spectral sequence above degenerates at $E^2$. In particular, it degenerates at $E^2$ when $X$ is a compact K\"ahler manifold by \cite{dgms}. When $X$ is a smooth variety (or the complement of a normal crossing divisor in a compact K\"ahler manifold), it is a spectral sequence of mixed Hodge structures as is its restriction to its set of primitive elements. (This also implies the degeneration statement for compact K\"ahler manifolds.)
\end{remark}

For future reference: the low degree terms of the $E^1$-term of the spectral sequence (with $\kk$ coefficients) are:
\begin{equation}
\label{eqn:E1}
\xymatrix@R=2ex{
H_2(X)^{\otimes 3} & \ar[l]_{\partial_1} [\H_\bdot(X)^{\otimes 2}]_6 & \ar[l]_(.35)\Delta H_6(X) & t=6 \cr
0 & [\H_\bdot(X)^{\otimes 2}]_5 & \ar[l]_(.4)\Delta H_5(X) & t=5 \cr
0 & H_2(X)^{\otimes 2} & \ar[l]_(.4)\Delta H_4(X) & t= 4 \cr
0 & 0 & H_3(X) & t = 3 \cr
0 & 0 & H_2(X) & t=2 \cr
s=-3 & s=-2 & s=-1 &
}
\end{equation}
where $\Delta : \H_\bdot(X;\kk) \to \H_\bdot(X;\kk)^{\otimes 2}$ is the reduced coproduct (i.e., the dual of the cup product),
$$
[\H_\bdot(X)^{\otimes 2}]_5 = H_2(X)\otimes H_3(X) + H_3(X)\otimes H_2(X)
$$
and
$$
[\H_\bdot(X)^{\otimes 2}]_6 =  H_3(X)\otimes H_3(X) + H_2(X)\otimes H_4(X) + H_4(X)\otimes H_2(X).
$$

The spectral sequence defines a filtration
$$
H_\bdot(L_{x_o};\kk) = A^0 H_\bdot(L_{x_o}) \supseteq A^1 H_\bdot(L_{x_o}) \supseteq A^2 H_\bdot(L_{x_o}) \supseteq \cdots
$$
whose $s$th graded quotient satisfies
$$
\Gr^s_A H_{t-s}(L_{x_o}X;\kk) \cong E^\infty_{-s,t}.
$$
We'll call this the {\em Adams filtration}. It restricts to a filtration of the Lie algebra
$$
\pi_{\bdot+1}(X,x_o)_\kk \cong PH_\bdot(L_{x_o}X;\kk).
$$
It has the following properties:
\begin{enumerate}

\item $\Gr_A^\bdot H_\bdot(L_{x_o} X;\kk)$ is the universal enveloping algebra of $\Gr_A^\bdot\pi_{\bdot+1}(X,x_o)_\kk$.

\item $\Gr_A^s \pi_\bdot(X,x_o)_\kk$ is a subquotient of $\H_\bdot(X;\kk)^{\otimes s}$.

\item the filtration is compatible with the bracket:
$$
[A^j \pi_\bdot(X,x_o)_\kk, A^k \pi_\bdot(X,x_o)_\kk] \subseteq A^{j+k} \pi_\bdot(X,x_o)_\kk.
$$

\item When $X$ is a compact K\"ahler manifold, $A^\bdot$ is the weight filtration $W_\bdot$ of the canonical MHS. More precisely, $W_{-m} = A^m$.

\item $\Gr_A^1 \pi_\bdot(X,x_o)_\kk$ is image of rational Hurewicz homomorphism
$$
h : \pi_\bdot(X,x_o)_\kk \to H_\bdot(X;\kk).
$$

\end{enumerate}

The following result gives the analogue for simply connected spaces of the exact sequence (\ref{eqn:extn}) that is used to define the classical Johnson homomorphism using the group algebra $\Z\pi_1(X,x_o) = H_0(L_{x_o} X)$. The second sequence is the analogue of the extension
$$
0 \to L^2 \pi_1(X,x_o)/L^3 \to \pi_1(X,x_o)/L^3 \to H_1(X) \to 0
$$
where $L^j \pi_1(X,x_o)$ denotes the $j$th term of the lower central series of $\pi_1(X,x_o)$. It occurs in \cite{ccm} where Carlson, Clemens and Morgan study the Hodge theory of $\pi_3$ of a simply connected 3-fold.

\begin{proposition}
\label{prop:exact_seqces}
There are exact sequences
$$
0 \to H_2(X;\kk)^{\otimes 2}/\im \Delta \to H_2(L_{x_o} X;\kk) \to H_3(X;\kk) \to 0
$$
and
$$
0 \to \Sym^2 H_2(X;\kk)/\im \Delta \overset{[\blank,\blank]}\to \pi_3(X,x_o)\otimes\kk \overset{h}{\to} H_3(X;\Q) \to 0
$$
where
$$
\Delta : H_4(X;\kk) \to \Sym^2 H_2(X;\kk) \subset H_2(X;\kk)^{\otimes 2}
$$
is dual to the cup product, $h$ is the Hurewicz homomorphism and ${[\blank,\blank]}$ is the Whitehead product.
\end{proposition}

When $H_2(X;\Z)$ and $H_3(X;\Z)$ are torsion free, the first sequence is exact with $\Z$ coefficients.

\begin{proof}
The spectral sequence (\ref{eqn:E1}) implies that, $A^3 H_2(L_{x_o}X;\kk) = 0$. It also implies that for the terms of total degree 2, $E^\infty = E^2$. The first sequence is
$$
0 \to A^2H_2(L_{x_o}) \to A^1 H_2(L_{x_o}) \to \Gr_A^1 H_2(L_{x_o}) \to 0.
$$
The second sequence is its restriction to the set of primitive elements of $H_\bdot(L_{x_o};\Q)$.
\end{proof}

\section{Johnson homomorphisms for Simply Connected Manifolds}
\label{sec:johnson}

In this section we use the action of the Torelli group $\ho T_X$ of a simply connected finite complex on the homology of its loop space $L_{x_o}X$ to define generalized Johnson homomorphisms. The main point is that $\ho T_X$ acts trivially on the graded quotients of the Adams filtration of $H_\bdot(L_{x_o}X;\Q)$.

Denote the group of homotopy classes of homotopy equivalence of $X$ by $\ho E_X$ and its image in $\Aut H_\bdot(X;\Q)$ by $\ho S_X$. Since $X$ is simply connected, $\ho E_X$ acts on $\pi_\bdot(X,x_o)$ and thus on $H_\bdot(L_{x_o};\kk)$. Denote the kernel of the homomorphism $\ho E_X \to \Aut H^\bdot(X;\kk)$ by $\ho T_X$.

The action of $\ho E_X$ on $H_\bdot(L_{x_o}X;\Q)$ induces a homomorphism $\ho E_X \to \Aut \pi_\bdot(X)\otimes\Q$. This action preserves the Adams filtrations and one has a diagram
$$
\xymatrix@C=12pt{
1 \ar[r] & \ho T_X \ar[r]\ar[d] & \ho E_X \ar[r]\ar[d] & \ho S_X \ar[r]\ar[d] & 1
\cr
1 \ar[r] & A^{-1} \End \pi_\bdot(X,x_o)_\Q \ar[r] & \End \pi_\bdot(X,x_o)_\Q \ar[r] & \Aut H_\bdot(X;\Q) 
}
$$
and therefore an $\ho S_X$ invariant homomorphism
\begin{equation}
\label{eqn:ho-tau}
\ho\tau_X : H_1(\ho T_X) \to \Gr^{-1}_A \End \pi_\bdot(X,x_o)_\kk.
\end{equation}

In degree 2 this produces the Johnson homomorphism (\ref{eqn:tau}):

\begin{corollary}
The group $\ho E_X$ acts on the sequences in Proposition~\ref{prop:exact_seqces}. This action induces an $\ho S_X$ invariant homomorphism
\begin{equation}
\label{eqn:invariant}
\ho T_X \to \Hom_\kk(H_3(X;\Q),\Sym^2 H_2(X;\Q)/\im \Delta).
\end{equation}
\end{corollary}

Here we encounter our first problem. If $M$ is a compact K\"ahler manifold with $b_2=1$, then $H_2(M;\Q)/\im \Delta$ vanishes as the cup product
$$
\Sym^2 H^2(M;\Q) \to H^4(M;\Q)
$$
is an isomorphism. So, in this case, this Johnson homomorphism vanishes. This does not mean, however, that $H_1(T_M;\Q)$ vanishes as the homomorphism $H_1(T_M;\Q) \to H_1(\ho T_M;\Q)$ may have a non-trivial kernel. This kernel is detected by the {\em distortion} of the Pontryagin classes, as I learned from Sullivan \cite[\S13]{sullivan}.

\subsection{Distortion of the Pontryagin classes}
\label{sec:distortion}

Suppose now that $M$ is a simply connected smooth closed manifold and that $\varphi : M \to M$ is a diffeomorphism of $M$ that represents a class in $T_M$. The mapping torus $M_\varphi$ of $M$ fibers over the circle.\footnote{We take $M_\varphi$ to be the quotient of $M\times [0,1]$ where $(\varphi(x),0)$ is identified with $(x,1)$.} Denote the projection by $\pi : M_\varphi \to M$. Since $\varphi$ acts trivially on the homology of $M$, the Wang sequence of $\pi$ is a collection of short exact sequences
$$
\xymatrix{
0 \ar[r] & H^{j-1}(M) \ar[r] & H^j(M_\varphi)\ar[r]^{\iota^\ast} & H^j(M) \ar[r] & 0
}
$$
where the first map is Poincar\'e dual to the map $\iota_\ast : H_\bdot(M) \to H_\bdot(M_\varphi)$ induced by the inclusion $\iota : M \to M_\varphi$ of a fiber. The Pontryagin class $p_k(M)$ has a natural lift to $H^{4k}(M_\varphi)$. Namely $p_k(M_\varphi)$. When $\varphi$ is homotopic (and therefore smoothly homotopic) to the identity, there is a second lift, that we now describe.

Suppose that $F : M\times [0,1] \to M$ is a smooth homotopy from $\varphi$ to $\id_M$. We may assume that it is the projection $M\times I \to M$ when $I= [0,\epsilon]$ and $[1-\epsilon,1]$ for some $\epsilon > 0$. In this case, $F$ induces a smooth homotopy equivalence
$$
\Fhat : M\times S^1 \to M_\varphi
$$
and a map
$$
\xymatrix{
0 \ar[r] & H^{j-1}(M) \ar[r]\ar@{=}[d] & H^j(M_\varphi)\ar[r]^{\iota^\ast}\ar[d]^{\Fhat^\ast} & H^j(M) \ar[r] \ar@{=}[d] & 0 \cr
0 \ar[r] & H^{j-1}(M) \ar[r]^(.45){\times \theta} & H^j(M\times S^1)\ar[r] & H^j(M) \ar[r] & 0
}
$$
of Wang sequences, where $\theta$ is the positive integral generator of $H^1(S^1)$. The map $\Fhat^\ast$ is an isomorphism. The preimage of $p_k(M\times S^1)$ under $\Fhat^\ast$ is a second lift of $p_k(M)$ to $H^{4k}(M)$.

\begin{definition}
The {\em distortion} of $p_k(M)$ under $F$ is the difference
$$
\deltatilde_k(F) = \Fhat^\ast p_k(M_\varphi) - p_k(M\times S^1) \in H^{4k-1}(M).
$$
\end{definition}

These can be assembled into the {\em distortion of $F$}, which is defined by
$$
\deltatilde(F) := \big(\deltatilde_1(F),\dots,\deltatilde_{\lfloor m/4 \rfloor}(F)\big) \in \bigoplus_{0 < 4k \le m} H^{4k-1}(M)
$$
where $m=\dim_\R M$. The distortion of an isotopy vanishes as, in that case, $\Fhat$ is a diffeomorphism.

To obtain an invariant of the mapping class of $\varphi$, we have to mod out by the subgroup of distortions of smooth homotopies from $\id_M$ to itself. This is the group $I$ of {\em indeterminacies}. Set
$$
D_M = \big[\bigoplus_{4k \le \dim_\R M} H^{4k-1}(M;\Q)\big]/I
$$
The indeterminacy of a diffeomorphism $\varphi$ that is homotopic to the identity via $F$ is defined to be the image of $\deltatilde(F)$ in $D_M$.

\begin{remark}
\label{rem:indeterminacy}
The element $\deltatilde(F)$ of $I$ of a smooth homotopy from $\id_M$ to itself can be computed as follows. Identify $H^\bdot(M\times S^1)$ with
$$
H^\bdot(M)\otimes H^\bdot(S^1) \cong (H^\bdot(M) \otimes 1) \oplus (H^{\bdot-1}\otimes \theta)
$$
where, as above, $\theta$ is the positive integral generator of $H^1(S^1)$. The total Pontryagin classes of $M$ and $M\times S^1$ are related by
$$
p(M\times S^1) = p(M) \otimes 1.
$$
The distortion of $F$ is determined by the equation
\begin{equation}
\label{eqn:distortion}
\deltatilde(F)\otimes \theta = \Fhat^\ast (p(M)\otimes 1) - p(M\times S^1).
\end{equation}
\end{remark}

We will regard $H^\bdot(M)$ as a left $\G_M$-module via $\psi : u \mapsto (\psi^{-1})^\ast u$.

\begin{proposition}
\label{prop:distortion}
If $M$ is a simply connected smooth manifold, there is a well defined homomorphism
$$
\delta : \ker\{\G_M \to \ho E_M\} \to D_M.
$$
It takes the mapping class of a diffeomorphism $\varphi$ to $\deltatilde(F) \bmod I$, where $F$ is any smooth homotopy from $\varphi$ to $\id_M$. Moreover, the left $\G_M$ action on $H^\bdot(M)$ induces an action on $D_M$: for all $\psi \in \G_M$
$$
\delta(\psi\varphi \psi^{-1}) = \psi\cdot\delta(\varphi).
$$
In particular, $T_M$ acts trivially on $D_M$ so that it is an $S_M$-module.
\end{proposition}

\begin{proof}
That $\delta(\varphi)$ is well defined follows from the definitions. To prove that it is a homomorphism is an exercise in the definitions. Now suppose that $\varphi$ is a diffeomorphism of $M$ that is homotopic to the identity and that $\psi$ is a arbitrary orientation preserving diffeomorphism. Suppose that $F$ is a smooth homotopy from $\varphi$ to $\id_M$. Then
$$
G := \psi \circ F \circ (\psi^{-1}\times \id)
$$
is a homotopy from $\psi\varphi\psi^{-1}$ to $\id_M$. The map
$$
\psihat : M_\varphi \to M_{\psi\varphi\psi^{-1}},\quad (x,t) \mapsto (\psi(x),t)
$$
is a well defined diffeomorphism and the diagram
$$
\xymatrix{
M\times S^1 \ar[r]^\Fhat\ar[d]_{\psi\times 1} & M_\varphi \ar[d]^\psihat \cr
M\times S^1 \ar[r]^{\widehat{G}} & M_{\psi\varphi\psi^{-1}}
}
$$
commutes, where $\widehat{G}$ is induced by $G$. Since $\varphi^\ast p(M) = p(M)$, we have
\begin{align*}
\psi^\ast \deltatilde(G) &= (\psi\times 1)^\ast \widehat{G}^\ast p(M_{\psi\varphi\psi^{-1}}) - (\psi^\ast p(M))\times 1
\cr
&= \Fhat^\ast\psihat^\ast p(M_{\psi\varphi\psi^{-1}}) - (p(M)\times 1)
\cr
&= \Fhat^\ast p(M_\varphi) - (p(M)\times 1)
\cr
&= \deltatilde(F)
\end{align*}
which implies that $\delta(\psi\varphi \psi^{-1}) = \psi\cdot\delta(\varphi)$.
\end{proof}

\begin{corollary}
The distortion group $D_M$ is contained in the center of $T_M$.
\end{corollary}

\subsection{Hypersurfaces and complete intersections}
\label{sec:ci_like}

In this section, $M$ will be a complex projective manifold and $\P^m$ will denote complex projective $m$-space. When $M$ is a smooth hypersurface or complete intersection of dimension $\ge 3$, the formulas for the distortion homomorphism simplify and also lift to $T_M$. The first observation is that all Pontryagin classes are multiples of powers of the hyperplane class.

For a smooth projective variety $M$ in $\P^N$, denote the class of a hyperplane section by $\w \in H^2(M)$.

\begin{lemma}
If $M$ is a smooth complete intersection in $\P^N$, then $p_j(M)$ is an integral multiple of $\w^{2j}$.
\end{lemma}

\begin{proof}
The result is proved by induction on the codimension of $M$ in $\P^N$. It holds for $\P^N$. Now suppose that $Y$ is a smooth complete intersection in $\P^N$ for which the result holds and that $M$ is the intersection of $Y$ with a hypersurface of degree $d$. Then the total Chern class of $M$ is
$$
c(M) = c(Y)(1+d\w)^{-1} \in H^\bdot(M;\Z).
$$
From \cite[Cor.~15.5]{milnor-stasheff} it follows that the Pontryagin classes $p_j(M)$ satisfy
$$
1 - p_1 + p_2 - p_3 + \cdots = c(M) c(M)^\vee
$$
where $c(M)^\vee$ is the total Chern class of the cotangent bundle of $M$. The inductive hypothesis implies that $p_j(M)$ is an integral multiple of $\w^{2j}$.
\end{proof}

This allows us to give a simpler formula for the distortion homomorphism for smooth complete intersections of dimension $\ge 3$ and to lift it to a homomorphism $T_M \to D_M$.\footnote{Kreck and Su made a similar construction in \cite[\S2]{kreck-su}.}

More generally, we work with smooth manifolds $M$ with a distinguished class $\w \in H^2(M)$ satisfying
\begin{equation}
\label{eqn:conditions}\tag{$\ast$}
b_1 = 0 \text{ and } p_k(M) \text{ is a rational multiple of }\w^{2k}.
\end{equation}
In the case of complete intersections, one takes $\w$ to be the class of a hyperplane section.

Suppose that $M$ satisfies these conditions and that $\varphi$ is a diffeomorphism that represents an element of $T_M$. The Wang sequence implies that
$$
H^2(M_\varphi) \to H^2(M)
$$
is an isomorphism. Denote the element of $H^2(M_\varphi)$ that corresponds to $\w$ by $\w_\varphi$. Since $p_k(M) = a_k \w^{2k}$, it has the natural natural lift $a_k\w_\varphi^{2k}$ to $H^{4k}(M_\varphi)$. 

Define the distortion of $\varphi$ to be the difference between the vectors
$$
\big(p_1(M_\varphi),p_2(M_\varphi),\dots \big) - \big(a_1 \w_\varphi^2,a_2 \w_\varphi^4,\dots\big) \in D_M.
$$

\begin{lemma}
This agrees with the previous definition when $\varphi$ is homotopic to the identity for varieties $M$ satisfying (\ref{eqn:conditions}) above. 
\end{lemma}

\begin{proof}
Suppose that $F$ is a homotopy from $\varphi$ to the identity. Since $H^1(M)$ vanishes, $H^2(M\times S^1) \cong H^2(M)$ and the isomorphism
$
\Fhat^\ast : H^2(M_\varphi) \to H^2(M\times S^1)
$
takes $\w_\varphi$ to $\w \otimes 1$. This implies that
\begin{align*}
\deltatilde_k(F) &= \Fhat^\ast p_k(M_\varphi) - p_k(M)\otimes 1 \cr
&= p_k(M_\varphi) - (\Fhat^\ast)^{-1} (p_k(M)\otimes 1) = p_k(M_\varphi) - a_k \w_\varphi^{2k}.
\end{align*}
\end{proof}

One can also check that, for $M$ satisfying (\ref{eqn:conditions}), the extended distortion map $\delta : H_1(T_M) \to D_M$ is an $S_M$-invariant homomorphism. Details are left to the reader.

\begin{lemma}
If $M$ is a compact K\"ahler manifold satisfying (\ref{eqn:conditions}), then the group of indeterminacies vanishes, so that
$$
D_M = \bigoplus_{2j < \dim_\C M} H^{4j-1}(M;\Q).
$$
\end{lemma}

\begin{proof}
Suppose that $F$ is a homotopy from $\id_M$ to itself. The assumptions imply that $H^2(M\times S^1)$ is spanned by $\w$. Since $\Fhat^\ast$ is a ring homomorphism, this implies that
$$
\Fhat^\ast \w^j \otimes 1 = \w^j \otimes 1.
$$
Since $p_j(M)$ is a multiple of $\w^{2j}$,  the distortion of $F$ vanishes by formula (\ref{eqn:distortion}).
\end{proof}

In particular, for simply connected K\"ahler 3-folds with $b_2=1$, the distortion homomorphism induces a homomorphism $H_1(T_M) \to H^3(M;\Q)$. This is a rational version of the Johnson homomorphism defined by Kreck and Su in \cite{kreck-su:2}. This clarifies why their Johnson homomorphism detects $H_1(T_M;\Q)$ for such manifolds, while the Johnson homomorphism (\ref{eqn:tau}) does not. It is because the later detects the kernel of the homomorphism
$$
H_1(T_M) \to H_1(\ho T_M).
$$

\subsection{Other Johnson homomorphisms}

These can be constructed as needed by computing the relevant part of $\Gr^1_A \pi_\bdot(X,x_o)$. The next simplest case is:

\begin{proposition}
If $X$ is simply connected and $b_2(X) = 1$, then
$$
0 \to \big(H_2(X;\Q)\otimes H_3(X;\Q)\big)/\im \Delta \to \pi_4(X,x_o)_\kk \to PH_4(X;\kk) \to 0
$$
is exact, where $\Delta$ is dual to the cup product $H^2(X)\otimes H^3(X) \to H^5(X)$ and $PH_\bdot(X;\kk)$ denotes the primitive homology classes.
\end{proposition}

\begin{corollary}
For such $X$, there is a $\ho S_X$ invariant homomorphism
$$
H_1(\ho T_X) \to \Hom_\kk(PH_4(X;\kk), \big(H_2(X;\Q)\otimes H_3(X;\Q)\big)/\im \Delta).
$$
\end{corollary}

\section{Sullivan's Results}
\label{sec:sullivan}

Sullivan \cite[\S13]{sullivan} provides some very useful general tools for understanding mapping class groups of simply connected closed manifolds. They combine results from surgery theory and rational homotopy theory.

Every simply connected finite complex $X$ can be ``localized at 0'' to obtain a space $X_\rat$ and a map $r : X \to X_\rat$ satisfying
\begin{enumerate}

\item the homotopy and reduced integral homology groups of $X_\rat$ are rational vector spaces,

\item the map $r$ induces isomorphisms
$$
r_\ast : \pi_\bdot(X)\otimes \Q \to \pi_\bdot(X_\rat)
\text{ and }
r_\ast : \Htilde_\bdot(X)\otimes \Q \to \Htilde_\bdot(X_\rat;\Z).
$$

\end{enumerate}
The space $X_\rat$ is determined (and is determined by) the Sullivan minimal model of $X$ or, dually, by the minimal DG Lie algebra model of $X$ that one can construct using Chen's power series connections. (See \cite{chen} and \cite{hain}.) Homotopy classes of automorphisms of the model correspond to homotopy classes of self maps of $X_\rat$. As a consequence, we have:

\begin{theorem}[Sullivan {\cite[Thm.~10.3]{sullivan}}]
\label{thm:sullivan_ho}
If $X$ is a simply connected finite complex, then $\ho E X_\rat$ is an affine algebraic $\Q$-group whose reductive quotient is a subquotient of the automorphism group of the rational cohomology ring $H^\bdot(X;\Q)$. Moreover the image of $\ho E_X \to \ho E_{X_\rat}$ is an arithmetic subgroup of $\ho E_{X_\rat}$; the kernel is finite. If, in addition, $X$ is a formal space, then the reductive quotient of $\ho E_{X_\rat}$ is the reductive quotient of the group of automorphisms of the cohomology ring $H^\bdot(X;\Q)$.
\end{theorem}

Denote the group of homotopy classes of homotopy equivalences of $M_\rat$ that fix the fundamental cohomology class of $M$ and all the Pontryagin classes $p_k(M) \in H^\bdot(M;\Q)$ by $\ho E_{M_\rat}^\pp$. It is an algebraic subgroup of $\Aut H^\bdot(X;\Q)$.

\begin{theorem}[Sullivan {\cite[\S13]{sullivan}}]
\label{thm:sullivan_diff}
If $M$ is a simply connected smooth closed manifold of real dimension $\ge 5$, then there is an affine algebraic $\Q$-group $G_M$ that is an extension
$$
0 \to D_M \to G_M \to \ho E_{M_\rat}^\pp \to 1
$$
where $D_M$ is the distortion group. There is homomorphism
$$
\rho : \G_M \to G_M(\Q)
$$
whose image is arithmetic (and so Zariski dense) and whose kernel is finite. When $M$ is formal, the reductive quotient of $G_M$ is the reductive quotient of the group of automorphisms of the ring $H^\bdot(M;\Q)$ that fix the Pontryagin classes of $M$.
\end{theorem}

Denote the kernel of $G_M \to \Aut H^\bdot(M;\Q)$ by $U_M$. This may be a proper subgroup of $G_M$.

\begin{corollary}
The homomorphism $\G_M \to G_M(\Q)$ induces a Zariski dense homomorphism
$$
T_M \to U_M(\Q)
$$
with finite kernel, which implies that the induced map $T_M^\un \to U_M$ is an isomorphism of the unipotent completion of $T_M$ with $U_M$.
\end{corollary}

Proposition~\ref{prop:distortion} implies that $D_M$ is central in $U_M$ and that the inclusion $D_M \hookrightarrow U_M$ induces an inclusion
\begin{equation}
\label{eqn:distortion_subgp}
D_M/(D_M\cap U_M') \hookrightarrow H_1(U_M) \cong H_1(T_M;\Q)
\end{equation}
where $U_M'$ denotes the commutator subgroup of $U_M$.

\begin{remark}
\label{rem:rel_comp}
Theorem~\ref{thm:sullivan_diff} can be restated in terms of relative unipotent completion, as defined in \cite{hain:rel_comp}. It says that $G_M$ is a quotient of the completion of $\G_M$ relative to the homomorphism from $\G_M$ onto its Zariski closure in $\Aut H^\bdot(X;\Q)$.
\end{remark}

One very basic example which will be used later is that of projective space.

\begin{example}
The Sullivan minimal model of (complex) projective space $\P^m$ is the commutative DGA
$$
M^\bdot_{\P^m} = \Q[a]\otimes \Lambda^\bdot(b)\quad db = a^{m+1}
$$
where $a$ has degree 2 and $b$ degree $2m+1$. The automorphism group is the multiplicative group $\Gm/\Q$ where $t\in \Gm(\Q)$ acts by $a\mapsto ta$ and $b \mapsto t^{m+1}b$. So $\ho E_{\P^m}^+ \cong \Gm$. The maximal arithmetical subgroup of $\Gm(\Q)$ that acts on $H^\bdot(\P^m;\Z)$ is the trivial group when $m$ is odd and $\{\pm 1\}$ when $M$ is even. The distortion group is trivial as there is no odd cohomology, so Sullivan's theorem implies that the mapping class group of $\P^m$ is finite. When $m$ is even, complex conjugation is orientation preserving and induces the non-trivial automorphism of the rational homotopy type.
\end{example}

\section{Cohomology Automorphisms of a Compact K\"ahler Manifold}
\label{sec:coho_autos}

The group of automorphisms of the rational cohomology ring of a compact oriented manifold that fix the Pontryagin classes can have a non-trivial unipotent radical. One simple example (pointed out to me by Alexander Berglund) is the unitary group $\mathrm{U}(9)$. It has trivial tangent bundle, and thus trivial Pontryagin classes. Its cohomology ring is an exterior algebra on generators $y_1,y_3,\dots,y_9$, where $y_j$ has degree $j$. The reductive quotient of its automorphism group is $\Gm$; a non-trivial unipotent automorphism takes $y_9$ to $y_1y_3y_5$.

Here we prove that the group of automorphisms of the cohomology ring of a compact K\"ahler manifold that fix the K\"ahler form is reductive, as is its subgroup that also fixes its Pontryagin classes.

\begin{proposition}
Suppose that $\kk$ is a subfield of $\R$. If $M$ is a compact K\"ahler manifold with K\"ahler class $\w \in H^2(M;\kk)$, then the automorphism group of its cohomology ring that fixes $\w$ is a reductive $\kk$ group.
\end{proposition}

Denote $\sL_2(\kk)$ by $\sL_2$. This has basis $\ee$, $\ff$ and $\hh$, where $[\hh,\ee] = 2\ee$, $[\hh,\ff] = -2\ff$ and $[\ee,\ff] = \hh$.

\begin{proof}
Denote the complex dimension of $M$ by $n$ and $H^\bdot(M;\kk)$ by $H^\bdot(M)$. The choice of a K\"ahler class $\w$ determines a unique action of $\sL_2$ on $H^\bdot(M)$, where $\ee$ acts as multiplication by $\w$ and $\hh$ acts on $H^{n+j}(M)$ by multiplication by $j$. The primitive cohomology $H^\bdot_o(M)$ is, by definition, the space of lowest weight vectors $\ker \ff$. The restriction of the pairing
$$
\bil : H^{n-j}(M)\otimes H^{n-j}(M) \to \kk,\quad \alpha\otimes \beta \mapsto \int_M \w^j \wedge \alpha\wedge \beta
$$
to $H_o^{n-j}(M)$ is (up to a sign) a polarization.

Denote the subgroup of the automorphism group of $H^\bdot(M)$ that fixes $\w$ by $G$ and its Lie algebra by $\g$. It consists of the (degree 0) derivations of $H^\bdot(M)$ that annihilate $\w$. Every $\delta\in \g$ commutes with the $\sL_2$ action and thus acts on the primitive cohomology and preserves the primitive decomposition
$$
H^m(M) = \bigoplus_{\ell\ge 0} \w^\ell \wedge H^{m-2\ell}_o(M)
$$
and the projections $H^m(M) \to H^{m-2\ell}_o(M)$.
Consequently, $\delta$ is determined by an endomorphism (of degree $\le 0$) of $H^\bdot_o(M)$. It follows that $\g$ is the kernel of the morphism
$$
\End^{\le 0} H_o^\bdot(M) \to H^2(M) \oplus \Hom(H^\bdot(M)^{\otimes 2},H^\bdot(M))
$$
that takes $\delta$ to $\big(\delta(\w),\mu\circ(\delta\otimes \id + \id\otimes \delta) - \delta\circ \mu\big)$. It follows that $\g$ is a polarized Hodge structure of weight 0 who polarization is invariant under the adjoint action. Standard arguments then imply that $\g$ has a compact real form, which implies that $G$ is reductive. We briefly recall the argument. If we set
$$
\p = \g_\R \cap \bigoplus_{p \text{ even}} \g^{p,-p} \text{ and } \k = \g_\R \cap \bigoplus_{p \text{ odd}} \g^{p,-p},
$$
then $\k$ and $\p$ are real Hodge substructures of $\g_\R$, $\g_\R = \k \oplus \p$ and
$$
[\k,\k]\subseteq \k,\ [\k,\p]\subseteq \p \text{ and } [\p,\p]\subseteq \k.
$$
The restriction of the polarization to $\k$ is positive definite and to $\p$ is negative definite. This implies that $\k \oplus i\p$ is the Lie algebra of a compact real form of $G(\C)$.
\end{proof}

If $V$ is a Hodge substructure of $H^\bdot(M)$, then Lie algebra
$$
\g_V := \{\delta \in \g : \delta|_V = 0\}
$$
is a polarized Hodge sub-structure of $\g$ and thus also has a compact form. Since the Pontryagin classes of $M$ are polynomials in the Chern classes of $M$, they are Hodge classes. Taking $V$ to be the span of the $p_j(M)$, we conclude:

\begin{corollary}
If $M$ is as in the proposition, then the group of automorphisms of its cohomology ring $H^\bdot(M;\kk)$ that fix the K\"ahler class $\w$ and the rational Pontryagin classes of $M$ is a reductive $\kk$ group.
\end{corollary}

The appropriate mapping class to consider when studying moduli of polarized K\"ahler or projective manifolds $(M,\w)$ should be the subgroup $\G_{M,\w}$ of $\G_M$ that fixes $\w$. Combining the corollary with Sullivan's Theorem~\ref{thm:sullivan_diff}, we obtain:

\begin{theorem}
\label{thm:kahler_coho}
Suppose that $\kk$ is a subfield of $\R$. If $(M,\w)$ is a compact K\"ahler manifold of dimension $\ge 3$ with $\w\in H^2(M;\kk)$, then the group $G$ of ring automorphisms of $H^\bdot(M;\R)$ that fix the K\"ahler class and the Pontryagin classes is a reductive $\kk$ group and the quotient $S_{M,\w}$ of $\G_{M,\w}$ by the Torelli group $T_M$ is a lattice in $G(\R)$.
\end{theorem}

\section{Proof of Theorem~\ref{thm:main}}

\subsection{Lie Algebra Models}

Sullivan's minimal model of a space $M$ is an algebraic model of the Postnikov tower of $M_\rat$. Dual to this is the minimal Lie algebra model, which corresponds to a minimal rational cell decomposition of $M_\rat$. Since compact manifolds are finite complexes, the Lie algebra model of a simply connected manifold is finitely generated and thus easier (for me) to work with. One construction is given by Chen's method of formal power series connections. The connection with rational cell decompositions is given in \cite{hain}.\footnote{A rational $n$-cell is the cone over a rational $(n-1)$ sphere $S^{n-1}_\rat$.}

\subsubsection{The DG-Lie algebra model of a finite complex}

Recall that a graded Lie algebra is a graded vector space $\g_\bdot$ with a graded bracket $\g_j\otimes \g_k \to \g_{j+k}$ satisfying
\begin{enumerate}

\item $[u,v] + (-1)^{|u||v|} [v,u] = 0$,

\item $(-1)^{|u||w|}[u,[v,w]] + (-1)^{|v||u|}[v,[w,u]] + (-1)^{|w||v|}[w,[u,v]] = 0$.

\end{enumerate}
Here $|u|$ denotes the degree $j$ of $u\in \g_j$. A DG-Lie algebra is a graded Lie algebra $\g_\bdot$ with a differential $\partial$ of degree $-1$, satisfying
$$
\partial [u,v] = [\partial u,v] + (-1)^{|u|}[u,\partial v].
$$

Every (say) connected finite complex $M$ has a DG-Lie algebra model $(\LL_M,\partial)$. It is a free Lie algebra $\LL(\ee_j)$ over $\Q$ with a differential $\partial$ of degree $-1$. These satisfy:
\begin{enumerate}

\item $\partial$ is minimal --- that is, $\partial \ee_j \in [\LL_M,\LL_M]$.

\item There is a natural isomorphism $(\LL_M/[\LL_M,\LL_M])_j \cong \H_{j+1}(M;\Q)$ --- equivalently, $\{X_j\}$ is a graded basis of $\H(M;\Q)$, with degrees reduced by 1.

\item There is a natural isomorphism $H_j(\LL_M,\partial) \cong \pi_{j+1}(M)\otimes \Q$.

\end{enumerate}
The diagram
$$
\xymatrix@C=32pt{
(\LL_M/[\LL_M,\LL_M])_{n-1} \ar[d]_\cong \ar[r] & ([\LL_M,\LL_M]/[\LL_M,\LL_M,\LL_M])_{n-2} \ar[d]^{\text{inclusion}} \cr
\H_n(M;\Q) \ar[r]^(.3){(J\otimes 1)\circ\Delta} & \bigoplus_{j+k=n} \H_j(M;\Q) \otimes \H_k(M;\Q)
}
$$
commutes, where $J : \H_j(X) \to \H_j(X)$ is multiplication by $(-1)^j$ and $\Delta$ is the standard coproduct. (That is, the map dual to the cup product.)

The Lie algebra model of $M$ is unique up to isomorphism. Automorphisms of $\LL_M$ induce homotopy classes of automorphisms of $M_\rat$.\footnote{This follows from the correspondence between the generators $\ee_j$ of $\LL_M$ and the rational $(j+1)$ cells in a minimal cell decomposition of $M_\rat$. The differential $\partial$ determines the attaching maps. See \cite[Thm.~4.14]{hain}.} which gives well defined homomorphism $\Aut(\LL_M,\partial) \to \ho E_{M_\rat}$.

A space $M$ is formal if and only if its Lie algebra model $\LL_M$ has a presentation where $\partial$ is quadratic. Lie algebra models of formal spaces are easy to write down as, up to a sign, the quadratic terms of $\partial$ are dual to the cup product in the cohomology ring. Since compact K\"ahler manifolds are formal \cite{dgms}, their Lie algebra models are quadratic.

\subsubsection{The Lie algebra model of a simply connected closed formal 6-manifold}

Suppose that $M$ is a simply connected closed formal manifold of real dimension 6. All homology, etc will be with $\Q$ coefficients. Choose a graded basis
\begin{align*}
&\{\ee_j\} \text{ of } H_2(M),\ j = 1,\dots,b_2\cr
&\{\zz_k\} \text{ of } H_3(M),\ |k| \le b_3/2,\ k\neq 0\cr
&\{\ff_j\} \text{ of } H_4(M),\ j = 1,\dots,b_2\cr
&\{\ww\} \text{ of } H_6(M)
\end{align*}
of $\H_\bdot(M)$. Poincar\'e duality allows us to choose this basis so that
$$
\Delta \ww = \sum_j (\ee_j \otimes \ff_j + \ff_j \otimes \ee_j) + \sum_{k>0} (\zz_k\otimes \zz_{-k} - \zz_{-k}\otimes \zz_k).
$$
Then $\LL_M$ is the free graded Lie algebra $\LL(\ee_j,\ff_m,\zz_k,\ww)$, where the $\ee_j$ have degree 1, the $\zz_k$ have degree 2, the $\ff_m$ have degree 3 and $\ww$ has degree 5. The differential is
$$
\partial \ee_j = 0,\ \partial \zz_k = 0,\ \partial \ff_m = \sum c_m^{i,j} [\ee_i,\ee_j] \text{ and } \partial \ww = \sum_j [\ee_j,\ff_j] -  \sum_{k>0} [\zz_k,\zz_{-k}].
$$

\subsection{Self Rational Homotopy Equivalences of a Simply Connected Closed Formal 6-manifold}

We continue with the notation of the previous section. We will construct an automorphism of $(\LL_M,\partial)$ as the exponential of a (graded, degree 0) derivation $D$ of $\LL_M$. Define $D$ by
$$
D \ee_j = 0,\ D \zz_k = \sum a_k^{s,t}[\ee_s,\ee_t],\ D \ff_m = \sum b_m^{j,k} [\ee_j,\zz_k],\ D \ww = 0.
$$
We have to show that the $a_k^{s,t}$ and $b_m^{j,k}$ can be chosen so that $D\partial = \partial D$.

\begin{lemma}
For all choices of $a_k^{s,t}$ and $b_m^{j,k}$, $[D,\partial] := D\partial - \partial D$ vanishes on all $\ee_j$, $\zz_k$ and $\ff_m$. It vanishes on $\ww$ if and only if
$$
b_j^{i,k} = -(a^{i,j}_{-k} + a^{j,i}_{-k}).
$$
\end{lemma}

\begin{proof}
The first assertion is easy to check. To prove the second, we have to determine when $D\partial \ww=0$. Equivalently we need to determine when
$$
D\big(\sum_{k>0}[\zz_k,\zz_{-k}]\big) = D\big(\sum_j [\ee_j,\ff_j]\big).
$$
The right hand side is
$$
-\sum_j \sum_{i,k} b_j^{i,k} [\ee_j,[\ee_i,\zz_k]].
$$
The left hand side is
\begin{align*}
\sum_{k>0}D[\zz_k,\zz_{-k}] &= \sum_{k>0}[D\zz_k,\zz_{-k}] + \sum_{k>0}[\zz_k,D\zz_{-k}]
\cr
&= \sum_{k>0}[D\zz_k,\zz_{-k}] + \sum_{k<0}[\zz_{-k},D\zz_k]
\cr
&= \sum_{k>0}[D\zz_k,\zz_{-k}] + \sum_{k<0}[D\zz_k,\zz_{-k}]
\cr
&= \sum_k [D\zz_k,\zz_{-k}]
\cr
&= \sum_{k,s,t} a_k^{s,t}[[\ee_s,\ee_t],\zz_{-k}]
\cr
&= \sum_{k,s,t}a_k^{s,t}\big([[\ee_s,\zz_{-k}],\ee_t] + [\ee_s,[\ee_t,\zz_{-k}]]\big)
\cr
&= \sum_{k,i,j}(a^{i,j}_{-k} + a^{j,i}_{-k})[\ee_j,[\ee_i,\zz_k]] 
\end{align*}
\end{proof}

Thus, for each choice of $a_j^{s,t}$ we have the automorphism $\exp D$ of $(\LL_M,\partial)$ and thus an automorphism of $M_\rat$ that induces the identity on $H_\bdot(M_\rat)$. For generic $a_j^{s,t}$, this automorphism is not homotopic to the identity provided that $b_2>1$ and $b_3>0$ as the invariant (\ref{eqn:invariant}) of the automorphism is given by
$$
\zz_j \mapsto \sum_{s,t} a_j^{s,t} \ee_s\ee_t \in S^2 H_2(M;\Q)/\im \Delta.
$$

\begin{corollary}
If $M$ is a simply connected closed formal 6-manifold, then the homomorphism
$$
\ho T_{M_\rat} \to \Hom(H_3(M;\Q),S^2 H_2(M;\Q)/\im\Delta)
$$
is surjective.
\end{corollary}

Theorem~\ref{thm:main} now follows directly from Sullivan's Theorem~\ref{thm:sullivan_diff}, as it implies that the unipotent completion of $T_M$ surjects onto $\ho T_{M_\rat}$ with kernel $D_M$.

\section{Moduli of Hypersurfaces}
\label{sec:hypersurfaces}

Teichm\"uller theory implies that if $M$ is a smooth projective curve of genus $g\ge 2$, then the monodromy homomorphism $\pi_1(\sM_g,[M]) \to \G_M$ is an isomorphism. Here $\sM_g$ denotes the moduli space of smooth projective curves of genus $g$ and $\pi_1(\sM_g,[M])$ its orbifold fundamental group. One can ask if his phenomenon persists in higher dimensions. Here, as a test case, we consider the test case of hypersurfaces in projective space.

The moduli stack $\sH_{n,d}$ of smooth hypersurfaces of degree $d$ in $\P^{n+1}$ will be regarded as the stack quotient
$$
\sH_{n,d} := \sU_{n,d}\ffs \GL_{n+2}(\C)
$$
of the space $\sU_{n,d}$ of homogeneous polynomials of degree $d$ in $n+2$ variables with non-vanishing discriminant by the natural action of $\GL_{n+2}$. Since $\pi_1(\GL_n(\C)) \cong \Z$, there is a central extension
\begin{equation}
\label{eqn:cent_ext}
0 \to \Z \to \pi_1(\sU_{n,d},f) \to \pi_1(\sH_{n,d},[M_f]) \to 1
\end{equation}
of orbifold (or stack) fundamental groups, where $M_f$ denotes the hypersurface in $\P^{n+1}$ defined by the homogeneous polynomial $f$.

Since there is a universal hypersurface over $\sU_{n,d}$, there is a monodromy representation
$$
\lambdahat_M : \pi_1(\sU_{n,d},f) \to \G_M.
$$
Since the restriction of the universal hypersurface to each $\C^\ast$ orbit is trivial, it induces homomorphisms
\begin{equation}
\label{eqn:monod}
\lambda_M : \pi_1(\sH_{n,d},[M]) \to \G_M \text{ and } \rho_M : \pi_1(\sH_{n,d},[M]) \to \Aut H^n_o(M;\Q),
\end{equation}
where $H^\bdot_o(M)$ denotes the primitive (in the sense of Hodge theory) cohomology of $M$.

\begin{remark}
\label{rem:middle_betti}
The dimension $r_{n,d}$ of $H^n_o(M;\Q)$ can be computed using the Lefschetz hyperplane theorem and the fact that $c_n(M)$ is the Euler class of $M$, so that
$$
r_{n,d} = (-1)^n \big(n+1-\chi(M)\big) = (-1)^n\Big(n+1-\int_M c_n(M)\Big).
$$
In particular, $r_{3,d} = d^4-5d^3+10d^2-10d+4$, which is positive for all $d>2$. In general $r_{n,d} > 0$ except when $n$ is even and $d=1$, and when $n$ is odd and $d\le 2$. This, and similar assertions for smooth complete intersections, are proved in \cite[Cor.~5]{ewing}.
\end{remark}

Recall from Section~\ref{sec:ci_like} that if $M$ is a hypersurface in projective space of dimension $\ge 3$, then the distortion homomorphism $\delta$ extends naturally to the entire Torelli group $T_M$. For a smooth hypersurface in $\P^{n+1}$ and degree $d$, set
$$
K_M = \ker\{\rho_M : \pi_1(\sH_{n,d}[M]) {\longrightarrow} \Aut H^n_o(M;\Q)\}.
$$
This maps to $T_M$.

\begin{theorem}
\label{thm:vanishing}
Suppose that $n\ge 3$ and $n\equiv 3 \bmod 4$. If $M$ is a smooth hypersurface in $\P^{n+1}$ of degree $d>1$, then
$$
\xymatrix{
K_M \ar[r]^{\lambda_M} & T_M \ar[r]^(.38)\delta & H^n(M;\Q)
}
$$
is trivial. If $d\ge 3$, then the image of $\lambda_M$ has infinite index in $T_M$.
\end{theorem}

The universal hypersurface $\sX$ is the subvariety of $\sU_{n,d}\times \P^{n+1}$ defined as the vanishing locus of the evaluation map $(h,x) \mapsto h(x)$. Denote the pullback of the hyperplane class of $\P^{n+1}$ to $\sX$ by $\w_\sX$.

\begin{lemma}
\label{lem:total_chern}
The total Chern class of $\sX$ is
$$
c(\sX) = (1+\w_\sX)^{n+2}(1+d\w_\sX)^{-1} \in H^\bdot(\sX;\Q).
$$
\end{lemma}

\begin{proof}
The first step is to observe (using \cite[3.2.17]{deligne-hodge2}) that, since $\sU_{n,d}$ is an open subset of $\C^{n+1}$, the mixed Hodge structure (MHS) on $H^\bdot(\sU_{n,d})$ satisfies $W_m H^m(\sU_{n,d})=0$ when $m>0$. This and K\"unneth imply that the projection $\pi:\sU_{n,d} \times \P^{n+1} \to \P^{n+1}$ induces an isomorphism
$$
\pi^\ast : H^m(\P^{n+1};\Q) \to W_m H^m(\sU_{n,d}\times \P^{n+1};\Q)
$$
for all $m\ge 0$. In particular, $W_2 H^2(\sU_{n,d}\times \P^{n+1};\Q) = \Q\pi^\ast\w$. This and restriction to a fiber imply that the class $\sX$ is
$$
[\sX] = d\pi^\ast\w \in H^2(\sU_{n,d}\times \P^{n+1}).
$$
Since $\sU_{n,d}$ is an open subset of an affine space, it has trivial tangent bundle. It follows that 
$$
c(\sU_{n,d} \times \P^{n+1}) = (1 + \pi^\ast\w)^{n+2}.
$$
The result follows as the class of the normal bundle of $\sX$ in $\sU_{n,d}\times \P^{n+1}$ is the restriction of $[\sX]$ to $\sX$, which is $d\w_\sX$. 
\end{proof}

\begin{proof}[Proof of Theorem~\ref{thm:vanishing}]
The only case where $\delta$ is defined is when $n \equiv 3 \bmod 4$. Write $n=4k-1$. Fix $f\in \sU_{n,d}$ and let $M$ be the corresponding hypersurface. Since $\pi_1(\sU_{n,d},f)$ surjects onto $\pi_1(\sH_{n,d},[M])$, it suffices to show that the $\delta$ vanishes on the kernel $\Khat_M$ of $\pi_1(\sU_{n,d},f) \to \Sp(H^n(M))$. Observe that, since $p(\sX) = c(\sX)c(\sX)^\vee$, Lemma~\ref{lem:total_chern} implies that $p_k(\sX) = a_k\w_\sX^{2k}$ for some $a_k\in \Q$.

Suppose that $\gamma_\varphi : (S^1,1) \to (\sU_{n,d},f)$ is a loop that represents $\varphi \in \Khat_M$. The pullback of the universal family $\sX \to \sU_{n,d}$ along $\gamma_\varphi$ is the mapping cylinder $M_\varphi$ of $\varphi$. Let $\varphitilde : M_\varphi \to \sX$ be the map that lies above $\gamma_\varphi$. We compute $\delta(\varphi)$ using the recipe in Section~\ref{sec:ci_like}. We have
$$
p_k(M_\varphi) = \varphitilde^\ast p_k(\sX) = a_k\varphitilde^\ast \w_\sX^{2k}.
$$
The class $\w_\varphi$ is $\varphitilde^\ast \w_\sX$. The recipe now implies that
$$
\delta(\varphi) = p_k(M_\varphi) - a_k\w_\varphi^{2k} = 0 \in H^n(M)
$$
where we are identifying $H^n(M)$ with its image in $H^{4k}(M_\varphi)$.

The second assertion follows as the image of $\lambda_M$ lies in $\ker\delta$, which has infinite index as $H^n(M;\Q)$ is non-zero when $d\ge 3$ by \cite{ewing}.
\end{proof}

Combining Theorem~\ref{thm:vanishing} with Theorem~\ref{thm:kreck-su} yields:

\begin{corollary}
\label{cor:intern_fte}
When $n=3$, the intersection of the image of
$$
\lambda_M : \pi_1(\sH_{3,d},[M]) \to \G_M
$$
with $T_M$ is finite.
\end{corollary}

\begin{remark}
This result implies that the connected component of the Zariski closure of the image of $\pi_1(\sH_{3,d},[M])$ in the affine group $G_M$ is isomorphic to $\Sp(H^3(M;\Q))$. This gives a canonical splitting of the surjection $G_M \to \Sp(H^3(M))$. One possible explanation of why there should be such a splitting is that, in complex dimension 3, there are 3 types of Dehn twists on embedded 3-spheres as explained by Kreck and Su \cite{kreck-su:2}. To see why, recall that a Dehn twist of $B^3 \times S^3$ is determined by the homotopy class of a map
\begin{equation}
\label{eqn:exceptional_splitting}
(B^3,\partial B^3) \to (S^3,\ast) \to (\SO(4),\id);
\end{equation}
that is, by an element of $\pi_3(\SO(4))$. Recall the exact sequence
$$
0 \to \pi_3(\SO(3)) \to \pi_3(\SO(4)) \to \pi_3(S^3) \to 0
$$
which is part of the homotopy exact sequence associated to the principal $\SO(3)$-bundle $\SO(4) \to S^3$. A generator of the kernel is the double cover $S^3 \to \SO(3)$ composed with the inclusion. A splitting of the sequence is induced by the inclusion $S^3 \cong \SU(2)$ of the unit quaternions into $\SO(4)$. Apart from spheres that bound a ball, there are two kinds of embedded 3-spheres: those that are non-trivial in homology and those that represent a generator of the kernel of the Hurewicz homomorphism $\pi_3(M) \to H_3(M)$, which is cyclic of order $d$.

The three kinds of Dehn twists are:
\begin{enumerate}

\item[(1)] Twists about embedded spheres that represent non-trivial classes in $H_3(M)$ by an element of $\pi_3(\SO(4))$ in the image of the splitting. These act as symplectic transvections on $H_3(M)$.

\item[(2)] Twists about embedded spheres that represent non-trivial classes in $H_3(M)$ by an element of $\pi_3(\SO(4))$ in the kernel. These lie in $T_M$.

\item[(3)] Twists about embedded spheres that are trivial in homology. These also lie in $T_M$.

\end{enumerate}
Picard--Lefschetz theory \cite{lamotke} implies that $\pi_1(\sH_{n,d},[M])$ is generated by Dehn twists associated to vanishing cycles of ordinary double points. These act as symplectic transvections on $H_3(M)$. It seems that the image of $\rho_M$ does not contain the Dehn twists of the second type.

Since $\pi_n(\SO(n+1))$ has rank 2 when $n\equiv 3 \bmod 4$ and rank 1 otherwise, there will be a similar trichotomy in these dimensions. These are precisely the $n$ for which $D_M$ is non-trivial.
\end{remark}

We now focus on the main result \cite[Thm.~1.2]{carlson-toledo} of Carlson and Toledo. The statement below is a weaker version which will suffice for our purposes.

\begin{theorem}[Carlson--Toledo]
\label{thm:carlson_toledo}
Suppose that $M$ is a smooth hypersurface of degree $d$ in $\P^{n+1}$. If $d\ge 3$ and $n>1$, the kernel of the representation $\pi_1(\sH_{n,d},[M]) \to \Aut H^n(M)$ surjects onto a lattice in a non compact, almost simple $\R$-group of rank $\ge 2$. In particular, it contains a non-abelian free subgroup.
\end{theorem}

Sullivan's Theorem~\ref{thm:sullivan_diff} implies that $T_M$ is an extension of a lattice in a unipotent group by a finite group. This implies that a homomorphism from a non-abelian free group to $T_M$ has infinite kernel.

\begin{corollary}
\label{cor:nab_free}
If $n\ge 3$ and $d\ge 3$, the kernel of $\lambda_M$ contains a non-abelian free group. In particular, it is infinite.
\end{corollary}

It is worth recalling the idea behind the result of Carlson--Toledo as it explains why the monodromy homomorphism (\ref{eqn:monod}) has a large kernel, which may be relevant in other contexts. The key observation is that one has a diagram
$$
\xymatrix{
\sU_{n,d} \ar[r]^\phi \ar[d]_\pi & \sH_{n+1,d} \cr
\sH_{n,d}
}
$$
where $\phi$ is induced by the map $\sU_{n,d} \to \sU_{n+1,d}$ that takes the degree $d$ homogeneous polynomial $f(x)$ to $y^d + f(x)$. If $M$ is the hypersurface $f(x) = 0$, then the hypersurface $\Mhat$ corresponding to $\phi([M])$ is the cyclic cover of $\P^{n+1}$ of degree $d$ branched along $M$.

The fundamental group of $\sU_{d,n}$ is the central extension (\ref{eqn:cent_ext}). This provides a second monodromy representation
$$
\rhohat_M : \pi_1(\sU_{n,d}) \to (\Aut H^{n+1}(\Mhat;\Q))/\text{scalars}
$$
which is quite different from the first. General Hodge theory implies that the Zariski closure of the image is reductive. Carlson and Toledo show that the image of $\ker\rho_M$ is a lattice in the Zariski closure of the image of $\rhohat_M$, which they show has a non-compact factor of real rank $\ge 2$.

This construction is analogous to, and reminiscent of, Looijenga's construction \cite{looijenga:prym} of ``Prym representations'' of mapping class groups in the surface case. It does not imply that $\pi_1(\sM_g,[M])\to \G_M$ has a large kernel as it does in the case of 3-folds as the classical Torelli group is not a lattice in a unipotent group. Rather, it is large enough and flexible enough to surject onto lattices in a countable number of non-isomorphic semi-simple $\Q$-groups even though it is far from being free.

\section{Open Problems and Future Directions}
\label{sec:future}

There are many interesting open questions about mapping class groups of simply connected smooth projective varieties and fundamental groups of the corresponding moduli spaces. Here we propose some that are related to the topics in this note.

\subsection{Topology}

These problems, while topological, should have implications for the geometry and topology of moduli spaces of simply connected projective manifolds.

\begin{problem}
Generalize Theorem~\ref{thm:main} and the computations of of Kreck--Su \cite{kreck-su:2} to larger classes of simply connected compact K\"ahler manifolds of dimension $\ge 3$ such as hypersurfaces of higher dimension and hyper K\"ahler manifolds.
\end{problem}

\begin{question}
Does the image of the monodromy homomorphism $\lambda_M : \pi_1(\sH_{n,d},[M]) \to\G_M$ have infinite index when $n\not\equiv 3 \bmod 4$? (See Theorem~\ref{thm:hypersurfaces}.)
\end{question}

\begin{problem}
Determine whether Corollary~\ref{cor:intern_fte} holds for all $n\ge  3$.
\end{problem}

\begin{question}
Is it true that for all simply connected 3-folds the sequence
\begin{multline*}
0 \to D_M/(D_M\cap U_M') \to H_1(T_M;\Q) \cr
\to \Hom(H_3(M;\Q),\Sym^2 H_2(M;\Q)/\im\Delta) \to 0
\end{multline*}
is exact, where the left hand map is the inclusion (\ref{eqn:distortion_subgp})?
\end{question}

This sequence is certainly a complex and is exact on the left and right. The problem is to determine whether $H_1(\ho T_M;\Q)$ is larger than the right-hand term. There are Hodge theoretic reasons to believe that $D_M\cap U_M'$ is trivial, so that $D_M$ injects into $H_1(T_M;\Q)$ as it does for in the case of hypersurfaces. See Question~\ref{quest:wt_distortion} below.
\medskip

To understand unipotent completions of Torelli groups, it is necessary to understand the group of homotopy classes of homotopy equivalences of a simply connected compact K\"ahler manifold that act trivially on its rational homotopy groups.

\begin{problem}
Understand the kernel of the map $\ho E_X \to \Aut \pi_\bdot(X,x_o)_\Q$ where $X$ is a finite CW complex which is formal in the sense of rational homotopy theory. Sullivan's results imply that it is a lattice in a unipotent group. For example, is it detected by some kind of generalized Johnson homomorphism?
\end{problem}

It is natural to ask how closely related fundamental groups of moduli spaces/stacks of smooth projective varieties are to the mapping class groups of the underlying diffeomorphism type. One basic question is:

\begin{question}
Suppose that one has a moduli space (or stack) $\sM$ of complex structures on the manifold underlying a simply connected smooth projective variety $M$. What are the constraints on the Zariski closure of the image of the monodromy representation $\pi_1(\sM,[M]) \to \G_M \to G_M(\Q)$? What can one say about the kernel of the monodromy representation, either in specific cases or in general?
\end{question}

We've already seen that for hypersurfaces of dimension congruent to 3 mod 4, the kernel is large (Carlson--Toledo) and the image is not Zariski dense as its intersection with $T_M$ is constrained by the vanishing of the distortion.
\medskip

The next few questions/problems are aimed at getting a better understanding of the relationship between $\pi_1(\sH_{n,d},[M])$ and mapping class groups. In particular, it seems important to get a better understanding the topological reasons for the failure of the (virtual) injectivity and surjectivity of the geometric monodromy homomorphism $\lambda_M$ in Theorem~\ref{thm:hypersurfaces}. One possible explanation for this failure is that we are considering the wrong mapping class group. A hypersurface $M$ in $\P^{n+1}$ can (and should?) be regarded as a pair $(\P^{n+1},M)$. This suggests that we should consider the mapping class group
$$
\G_{(\P^{n+1},M)} := \pi_0 \Diff^+(\P^{n+1},M)
$$
where $\Diff^+(\P^{n+1},M)$ denotes the group of orientation preserving diffeomorphisms of $\P^{n+1}$ that restrict to a diffeomorphism of $M$. The geometric monodromy of the universal family $\sX \subset \sH_{n,d}\times \P^{n+1}$ over $\sH_{n,d}$ is a homomorphism
$$
\lambda_{(\P^{n+1},M)} : \pi_1(\sH_{n,d},[M]) \to \G_{(\P^{n+1},M)}
$$

\begin{question}
How close is $\lambda_{(\P^{n+1},M)}$ to being an isomorphism?
\end{question}

\begin{problem}
Develop tools for understanding $\G_{(\P^{n+1},M)}$. For example, can the results of Sullivan in Section~\ref{sec:sullivan} generalize to this relative setting?
\end{problem}

As in Section~\ref{sec:hypersurfaces}, $\pi : \Mhat \to \P^{n+1}$ is the cyclic covering $y^d+f(x)=0$ of $\P^{n+1}$ branched along $M$, where $M$ is the hypersurface defined by the homogeneous polynomial $f$. The automorphism group of $\pi$ is the group $\bmu_d$ of $d$th roots of unity. Set
$$
\Diff_\pi \Mhat = \{\varphi \in \Diff^+ \Mhat: \varphi \text{ commutes with the } \bmu_d \text{ action}\}
$$
and $\G_{\Mhat,\pi} := \pi_0 \Diff_\pi \Mhat$. There is a central extension
$$
1 \to \bmu_d \to \G_{\Mhat,\pi}  \to \G_{(\P^{n+1},M)}\to 1
$$
and a monodromy homomorphism $\G_{(\P^{n+1},M)} \to \G_\Mhat$. These fit into a commutative diagram
$$
\xymatrix{
0 \ar[r] & \Z \ar[r] \ar[d] & \pi_1(\sU_{n,d},f) \ar[r]\ar[d] & \pi_1(\sH_{n,d},[M]) \ar[d]_{\lambda_{(\P^{n+1},M)}} \ar[dr]^{\lambda_M} \cr
1 \ar[r] & \bmu_d \ar[r] & \G_{\Mhat,\pi} \ar[r] \ar[d] & \G_{(\P^{n+1},M)} \ar[r] & \G_M \cr
&& \G_\Mhat
}
$$
where the composition of the two vertical maps is the geometric monodromy $\lambdahat_M$ of the family $\sXhat$ over $\sU_{n,d}$.

\begin{question}
Is $\pi_1(\sU_{n,d},f)$ mod a finite group commensurable with the product of its images in $\G_M$ and $\G_\Mhat$?
\end{question}

If so and if one believes that Kreck--Su generalizes to higher dimensions, this will imply that $\pi_1(\sU_{n,d},f)$ mod a finite group is commensurable with the product of is images in the cohomology rings of $M$ and $\Mhat$, which would imply that, after quotienting by a finite group, it is linear and residually finite. This leads us to the following question of Carlson and Toledo.

\begin{question}[Carlson--Toledo { \cite[p.~647]{carlson-toledo}}]
Are the groups $\pi_1(\sH_{n,d},[M])$ linear? Are they residually finite?
\end{question}

\subsection{Hodge theory}

Hodge theory is a potentially useful tool for investigating mapping class groups of compact K\"ahler manifolds. One possible medium for applying it is the Hodge theory of relative unipotent completion \cite{hain:malcev}.

There should be an interesting and useful Hodge theory for the relative completion of $\G_M$ when $M$ is compact K\"ahler. As pointed out in Remark~\ref{rem:rel_comp}, the group $G_M$ is a quotient of the relative completion of $\G_M$ with respect to the monodromy homomorphism $\G_M \to \Aut H^\bdot(M;\Q)$.

When $M$ is a simply connected compact K\"ahler manifold, $\pi_\bdot(M,x_o)_\Q$ has a natural MHS which does not depend on $x_o$. This implies that the coordinate ring of $\Aut \pi_\bdot(M,x_o)_\Q$ is a Hopf algebra in the category of MHS. Recall from Section~\ref{sec:adams} that the weight filtration on $\pi_\bdot(M,x_o)$ is related to the Adams filtration by $A^m \pi_\bdot(M) = W_{-m}\pi_\bdot(M)$. This implies that the weights on the Lie algebra of the Zariski closure of $T_M$ in $\Aut \pi_\bdot(M,x_o)_\Q$ are strictly negative.

\begin{conjecture}
If $M$ is a simply connected compact K\"ahler manifold, the coordinate ring and Lie algebra of the group $G_M$ in Theorem~\ref{thm:sullivan_diff} should have natural MHSs such that the homomorphisms
$$
D_M \to G_M  \text{ and } G_M \to \Aut \pi_\bdot(M,x_o)_\Q
$$
are morphisms of MHS. (That is, the induced map on coordinate rings is a morphism of MHS.) When $M$ is algebraic, these should form admissible variations of MHS over the moduli space $\sM$ that parameterizes algebraic structures on the manifold underlying $M$. The weights on the Lie algebra $\u_M$ of the unipotent radical of $M$ should be strictly negative.
\end{conjecture}

The local systems $R^j \pi_\ast \Q$ over $\sM$ associated to the universal family $\pi:\sX \to \sM$ are polarized variations of Hodge structure (PVHS) over $\sM$. By the results of \cite{hain:malcev}, the completion of $\pi_1(\sM,[M])$ with respect to the monodromy representation of the variation
$$
\bH := \bigoplus_j R^j \pi_\ast \Q
$$
has a canonical MHS. Denote this relative completion by $\cG_M$. The local system over $\sM$ whose fiber over the moduli point of $M$ is the coordinate rings of the $\cG_M$ is an admissible variation of MHS.

\begin{conjecture}
The homomorphism $\cG_M \to G_M$ induced by the monodromy representation $\pi_1(\sM,[M]) \to \G_M$ is a morphism of MHS.
\end{conjecture}

The MHS on the distortion group should have weight $-1$ as it is a subquotient of the weight $-1$ Hodge structure
$$
\bigoplus_j \Hom(\Q\, p_j(M),H^{4j-1}(M)) \cong \bigoplus_j H^{4j-1}(M)(2j).
$$

\begin{question}
\label{quest:wt_distortion}
Assuming that there is a natural MHS on $\u_M$, is the induced MHS on $D_M$ pure of weight $-1$?
\end{question}

If this is the case, it will imply that $D_M \to H_1(T_M;\Q)$ is injective when $M$ is compact K\"ahler as it implies that that $D_M\cap [\u_M,\u_M]$ vanishes.
\begin{question}
What is $\Gr^W_{-1} H_1(\u_M)$? And what are the possible weights on $H_1(\u_M)$?
\end{question}

In the case of curves, $H_1(\u_M)$ is pure of weight $-1$ when $g\ge 3$, \cite{hain:torelli}; is pure of weight $-2$ when $g=2$, \cite{watanabe}; and has arbitrarily negative weights when $g=1$, \cite[\S12]{hain:modular}. These questions are relevant as $\Gr^W_1H^1(\u_M;\Q)$ controls the normal functions defined over $\sM$: Suppose that $\V$ is a PVHS over $\sM$ of weight $-1$ with fiber $V$ over $[M]$. Denote the corresponding bundle of intermediate jacobians over $\sM$ by $\cJ(\V)$ and the space of normal function sections of it by $H^0(\sM,\cJ(\V))$. There are canonical isomorphisms
$$
H^0(\sM,\cJ(\V)) \cong \G H^1(\sM,\V) \cong \G\Hom(\Gr^W_{-1} H_1(\u_M,V))^{\pi_1(\sM,[M])},
$$
where $\G$ denotes $\Hom_\MHS(\Q(0),\blank)$, which picks out the Hodge classes of a MHS of weight 0.

\end{document}